\documentclass[12pt]{article}
\usepackage{booktabs}
\usepackage{caption}
\usepackage{mathrsfs}
\usepackage{amsmath}
\usepackage{amsfonts,amsthm,amssymb,mathrsfs,bbding}
\usepackage{graphics,multicol}
\usepackage{graphicx}
\usepackage{color}
\usepackage{enumerate}
\usepackage{caption}
\usepackage{rotating}
\usepackage{lscape}
\usepackage{longtable}
\allowdisplaybreaks[4]
\usepackage{tabularx}
\usepackage[%dvipdfm,
colorlinks,
anchorcolor=blue,
linkcolor=blue,
citecolor=blue]{hyperref}
\usepackage{extarrows}
\usepackage{cite}
\usepackage{latexsym,bm}
\usepackage{mathtools}
\evensidemargin=\oddsidemargin\topmargin=-1.5cm

\newtheorem{thm}{Theorem}

\newtheorem{lem}{Lemma}
\newtheorem{cor}{Corollary}

\newtheorem{remark}{Remark}

\theoremstyle{definition}

\addtocounter{section}{0}
\begin{document}

\title{\bf The spectrum and automorphism group of the set-inclusion graph}
\author{{Xueyi Huang$^a$, Qiongxiang Huang$^b$\footnote{Corresponding author.}\setcounter{footnote}{-1}\footnote{\emph{E-mail address:} huangxy@zzu.edu.cn (X. Huang), huangqx@xju.edu.cn (Q. Huang), jfwang@sdut.edu.cn (J. Wang).}}, Jianfeng Wang$^c$\\[2mm]
\small $^a$School of Mathematics and Statistics, Zhengzhou University, \\
\small Zhengzhou, Henan 450001, P.R. China\\
\small  $^b$College of Mathematics and Systems Science, Xinjiang University, \\
\small  Urumqi, Xinjiang 830046, P.R. China\\
\small  $^c$School of Mathematics and Statistics, Shandong University of Technology, \\
\small  Zibo, Shandong 255049, P.R. China}

\date{}
\maketitle
{\flushleft\large\bf Abstract}
Let $n$, $k$ and $l$ be integers with $1\leq k<l\leq n-1$. The set-inclusion graph $G(n,k,l)$ is the graph whose vertex set consists of all $k$- and $l$-subsets of $[n]=\{1,2,\ldots,n\}$, where two distinct vertices  are adjacent if one of them is contained in another. In this paper, we determine the spectrum and  automorphism group of $G(n,k,l)$, respectively. 

\begin{flushleft}
\textbf{Keywords:}  set-inclusion graph, spectrum, automorphism group.
\end{flushleft}
\textbf{AMS Classification:} 05C50, 05C25

\section{Introduction}\label{s-1}
Let $G$ be a simple undirected graph with vertex set $V$ and edge set $E$. As usual, we denote by $\mathcal{L}(G)$ and $A(G)$ the \emph{line graph} and \emph{adjacency matrix} of $G$, respectively. The \emph{characteristic polynomial} of $G$ is defined as $\phi(G,x)=|x\cdot I-A(G)|$, where $I$ denotes the identity matrix of order $|V|$. The zeros of $\phi(G,x)$ are  called the \emph{eigenvalues} of $G$, and all eigenvalues  together with their multiplicities are called the \emph{spectrum} of $G$. In particular,  $G$ is said to be \emph{integral} if its spectrum consists entirely of  integers.   Also, if $G$ has a bipartition $V=V_1\cup V_2$ ($|V_i|=n_i$) such that both $V_1$ and $V_2$ are independent sets and  each vertex of $V_1$ (resp. $V_2$) has $r_1$ (resp. $r_2$) neighbors in $V_2$ (resp. $V_1$), then we say that  $G$ is a  \emph{semi-regular bipartite graph}  with parameters $(n_1,n_2,r_1,r_2)$.

An \emph{automorphism} of $G$ is a permutation on $V$ that preserves adjacency relations. The \emph{automorphism group} of $G$, denoted by $\mathrm{Aut}(G)$, is the set of all automorphisms of $G$. The graph $G$ is called \emph{vertex-transitive}, \emph{edge-transitive} and  \emph{arc-transitive} if $\mathrm{Aut}(G)$ acts transitively on its vertices, edges and arcs, respectively. As usual,  we view each edge of $G$ as a pair of oppositely directed arcs.

% if, for any two vertices $u,v\in V$, there exists some $\sigma\in \mathrm{Aut}(G)$ such that $u^\sigma=v$, and \emph{edge-transitive} (resp. \emph{arc-transitive}) if, for any two edges  $\{u_1, v_1\},\{u_2,v_2\}\in E$, there exists some $\sigma\in \mathrm{Aut}(G)$ (resp. $\sigma_1,\sigma_2\in\mathrm{Aut}(G)$) such that $\{u_1,v_1\}^\sigma=\{u_1^\sigma,v_1^\sigma\}=\{u_2,v_2\}$ (resp. $u_1^{\sigma_1}=u_2$,  $v_1^{\sigma_1}=v_2$, and $u_1^{\sigma_2}=v_2$,  $v_1^{\sigma_1}=u_2$). It is clear that $G$ is edge-transitive if and only if $\mathcal{L}(G)$ is vertex-transitive.

For any fixed integers $n$, $k$ and $l$ with $1\leq k<l\leq n-1$, the \emph{set-inclusion graph} $G(n,k,l)$ is defined as the graph with vertex set $V=\{v\mid v\subseteq [n],~|v|\in \{k,l\}\}$
and edge set $E=\{\{u,v\}\mid u,v\in V,~u\subset v~\mbox{or}~u\subset v\}$ (cf. \cite{BKT}). It is not difficult to verify that the mapping 
$
\tau: v\mapsto \bar{v}=[n]\setminus v
$
builds up an isomorphism between $G(n,k,l)$ and $G(n,n-l,n-k)$. Thus we may always assume that $k+l\leq n$ (thus $k\leq \frac{n-1}{2}$ due to $k<l$). By definition, we see that $G(n,k,l)$ is a connected semi-regular bipartite graph with parameters $\left(\binom{n}{k},\binom{n}{l},\binom{n-k}{l-k},\binom{l}{k}\right)$, where the corresponding  bipartition is  given by $V=V_1\cup V_2$ with
\begin{equation}\label{equ-1}
V_1=\{v\mid v\subseteq [n],~|v|=k\}~\mbox{and}~V_2=\{v\mid v\subseteq [n],~|v|=l\}.
\end{equation}
Thus the line graph $\mathcal{L}(G(n,k,l))$  is a connected  $\left(\binom{n-k}{l-k}+\binom{l}{k}-2\right)$-regular graph on $\binom{n-k}{l-k}\binom{n}{k}=\binom{l}{k}\binom{n}{l}$ vertices. 

For some special $k$ and $l$, the set-inclusion graph $G(n,k,l)$ has  received particular attention by various researchers. For example, Badakhshian, Katona and  Tuza \cite{BKT} determined the domination number of $G(n,1,l)$,  and gave lower and upper estimates on the domination number of $G(n,2,l)$. For $l=k+1$, the graph $G(n,k,k+1)=B(n,k)$ is just  the subgraph of the \emph{hypercube} $Q_n$ induced by  all $k$- and $(k+1)$-subsets of $[n]$, here  $Q_n$ can be viewed as  the graph with  all subsets of $[n]$ as its vertex set  in which   two vertices (subsets) are adjacent if their symmetric difference has precisely one element.  Particularly, the graph $B(2k+1,k)$ is known as the \emph{middle layer graph} \cite{MW} or the \emph{regular hyper-star graph} \cite{LCKK}. The famous  middle levels conjecture, which probably originated with Havel\cite{HA} and Buck and Wiedemann \cite{BW}, and  has also been attributed to Dejter, Erd\H{o}s, Trotter, and various others  \cite{KT}, asserts that $B(2k+1,k)$ is hamiltonian. In \cite{MU}, M\"{u}tze  completely confirmed the conjecture. In addition, the graph $G(n,k,n-k)=H(n,k)$  is  known as the \emph{bipartite Kneser graph} \cite{MS}, which is just the double cover of the Kneser graph (see Section \ref{s-2} for the definition). It was conjectured independently by Simpson \cite{SI} and Roth (see \cite{GO,HU}), that  $H(n,k)$ (notice that $H(2k+1,k)=B(2k+1,k)$) is hamiltonian \cite{SI}. Very recently,  M\"{u}tze and Su \cite{MS} confirmed this conjecture.

To study the spectrum and automorphism group of a graph is a basic problem in algebraic graph theory. In this paper, we focus on determining the spectrum and automorphism group  of the set-inclusion graph  $G(n,k,l)$. First of all, we  give the following theorem.

\begin{thm}\label{set-spec-1}
Let $n$, $k$ and $l$ be integers with $1\leq k<l\leq n-1$ and $k+l\leq n$. Then the spectrum of $G(n,k,l)$ is given by
$$
\left\{\left[\pm\sqrt{\beta_s}\right]^{\binom{n}{s}-\binom{n}{s-1}}\mid 0\leq s\leq k\right\}\cup \left\{\left[0\right]^{\binom{n}{l}-\binom{n}{k}}\right\},
$$
where
\begin{equation}\label{equ-2}
\beta_s=\sum_{i=\max\{2k-l,0\}}^{k}\binom{n-2k+i}{l-2k+i}\cdot\left(\sum_{r=0}^s(-1)^{s-r}\binom{s}{r}\binom{k-r}{i-r}\binom{n-k-s+r}{k-i-s+r}\right).
\end{equation}
\end{thm}
As an application, we  give the spectra of the line graph $\mathcal{L}(G(n,k,l))$, the graph $B(n,k)$ and its line graph $\mathcal{L}(B(n,k))$, which generalizes a result due to  Mirafzal \cite{M1} on the eigenvalues of $\mathcal{L}(B(n,1))$ (see Section \ref{s-2}).

Next we present the automorphism group of $G(n,k,l)$.

\begin{thm}\label{set-aut-1}
Let $n$, $k$ and $l$ be integers with $1\leq k<l\leq n-1$ and $k+l\leq n$. Then the automorphism group of $G(n,k,l)$ is 
$$\mathrm{Aut}(G(n,k,l))\cong 
\left\{
\begin{array}{ll}
S_n& \mbox{if $k+l<n$},\\
S_n\times \mathbb{Z}_2 & \mbox{if $k+l=n$},
\end{array}
\right.
$$
where $S_n$ is the symmetric group on $[n]$ and $\mathbb{Z}_2$ is the cyclic group of order $2$.
\end{thm}

As corollaries, we obtain the automorphism groups of the graph $B(n,k)$ and the bipartite Kneser graph $H(n,k)$, which have been given by  Mirafzal \cite{M3,M2} (see Section \ref{s-3}).

%Putting $l=k+1$ and $n-k$  in Theorem \ref{thm-3}, we obtain the following two results due to Mirafzal \cite{M1,M2}.
%
%\begin{cor}\label{cor-4}\emph{(Mirafzal \cite{M2})}
%Let $n$ and $k$ be integers with $1\leq k\leq \frac{n-1}{2}$. Then  the automorphism group of $B(n,k)$ is 
%$$\mathrm{Aut}(B(n,k))\cong
%\left\{
%\begin{array}{ll}
%S_n& \mbox{if $k<\frac{n-1}{2}$},\\
%S_n\times \mathbb{Z}_2 & \mbox{if $k=\frac{n-1}{2}$},
%\end{array}
%\right.
%$$
%where $S_n$ is the symmetric group on $[n]$ and $\mathbb{Z}_2$ is the cycle group of order $2$.
%\end{cor}
%
%\begin{cor}\label{cor-5}\emph{(Mirafzal \cite{M1})}
%Let $n$ and $k$ be integers with $1\leq k\leq \frac{n-1}{2}$. Then  the automorphism group of $H(n,k)$ is
%$$\mathrm{Aut}(H(n,k))\cong S_n\times \mathbb{Z}_2,
%$$
%where $S_n$ is the symmetric group on $[n]$ and $\mathbb{Z}_2$ is the cycle group of order $2$.
%\end{cor}

%In \cite{M1}, Mirafzal showed that the line graph $\mathcal{L}(B(n,1))$ ($n\geq 4$) is an integral graph with distinct eigenvalues $-2$, $-1$, $0$, $n-2$, $n-1$ by using the equitable partition method in algebraic graph theory. In this note, we determine the spectra of $G(n,k,l)$ and $\mathcal{L}(G(n,k,l))$ for arbitrary $k,l$ ($0\leq k<l\leq n$ and $k+l\leq n$). As by-products, we obtain the spectra of $B(n,k)$ and $\mathcal{L}(B(n,k))$ for arbitrary $k$, which generalizes the main result of Mirafzal \cite{M1}. In particular, we see that $\mathcal{L}(B(n,k))$ is also an integral graph.

\section{The spectrum of $G(n,k,l)$}\label{s-2}

Let $n$, $k$ and $i$ be  integers with  $0\leq i\leq k\leq \frac{n}{2}$. Denote by $J(n,k,i)$  the graph whose vertex set consists of all $k$-subsets of $[n]$, where two vertices ($k$-subsets) are adjacent if their intersection has size $i$ (here we regard $J(n,k,k)$ as the graph consisting of $\binom{n}{k}$ isolated vertices with each of them equipping with one loop).  In particular, the graphs $J(n,k)=J(n,k,k-1)$ and $K(n,k)=J(n,k,0)$ are known as the \emph{Johnson graph} and \emph{Kneser graph},  respectively.

Let $A_i$ be the adjacency matrix of $J(n,k,i)$ for $0\leq i\leq k$. It is well known that the set of matrices $\{A_i:0\leq i\leq k\}$ forms a symmetric association scheme  (see Knuth\cite{KN}, Bannai \& Ito \cite{BI}), which is called \emph{Johnson scheme}, since these matrices have the  properties: $A_k=I$, $\sum_{i=0}^k A_i=J$, $A_i^T=A_i$  and 
$$
A_iA_j=A_jA_i=\sum_{s=0}^k\left[\sum_{r=0}^s\binom{s}{r}\binom{k-s}{i-r}\binom{k-s}{j-r}\binom{n-2k+s}{k-i-j+r}\right]A_s
$$
for  $0\leq i\neq j\leq k$. 

The following lemma provides the  eigenvalues  of  $J(n,k,i)$ for each $i$.

\begin{lem}\label{Johnson-spec}\emph{(Knuth\cite{KN})}
Let $n$, $k$ and $i$ be  integers with $0\leq i\leq k\leq \frac{n}{2}$. Then the spectrum of $J(n,k,i)$ is given by
$$
\left\{
[\alpha_{i,s}]^{\binom{n}{s}-\binom{n}{s-1}}\mid 0\leq s\leq k
\right\},
$$
where
\begin{equation}\label{equ-3}
\alpha_{i,s}=\sum_{r=0}^s(-1)^{s-r}\binom{s}{r}\binom{k-r}{i-r}\binom{n-k-s+r}{k-i-s+r}.
\end{equation}
Moreover, for any fixed $s$, there exist $\left(\binom{n}{s}-\binom{n}{s-1}\right)$'s linearly independent vectors which are the common eigenvectors of $A_i$'s  (here $A_i$ is the adjacency matrix of $J(n,k,i)$) with respect to the eigenvalue $\alpha_{i,s}$ for $0\leq i\leq k$.
\end{lem}

Using Lemma \ref{Johnson-spec}, we now give the proof of Theorem \ref{set-spec-1}.

{\flushleft \bf Proof of Theorem \ref{set-spec-1}.}
Since $G(n,k,l)$ is  bipartite, its adjacency matrix can be written  as
$$A=\left[
\begin{matrix}
0 &B\\
B^T& 0\\
\end{matrix}
\right]
\begin{matrix}
V_1\\
V_2\\
\end{matrix},
$$
where $V_1$ and $V_2$ are given in (\ref{equ-1}).  Note that $|V_1|=\binom{n}{k}$ and $|V_2|=\binom{n}{l}$. We have
\begin{equation}\label{equ-4}
\left|xI-A^2\right|=\left|xI-BB^T\right|\cdot\left|xI-B^TB\right|=x^{\binom{n}{l}-\binom{n}{k}}\cdot \left|xI-BB^T\right|^2,
\end{equation}
where $\binom{n}{k}\leq \binom{n}{l}$ due to $k<l$ and $k+l\leq n$. Note that the spectrum of $G(n,k,l)$ is symmetric about $0$. If  $\xi_1,\xi_2,\ldots,\xi_{\binom{n}{k}}$  are all eigenvalues of $BB^T$, then  from (\ref{equ-4}) we can deduce that  the eigenvalues of $G(n,k,l)$ are $\pm\sqrt{\xi_1},\pm\sqrt{\xi_2},\ldots,\pm\sqrt{\xi_{\binom{n}{k}}}$ and  $\left(\binom{n}{l}-\binom{n}{k}\right)$'s $0$. Thus it suffices to determine the eigenvalues of $BB^T$.

 Recall that $V_1$ and $V_2$ consist of all $k$- and $l$-subsets of $[n]$, respectively. For any two vertices $u,v\in V_1$ (not necessarily distinct), we see that  $(BB^T)_{u,v}$ is just the number of common neighbors of $u,v$ in $V_2$. Then, according to the definition of $G(n,k,l)$, we obtain
$$
(BB^T)_{u,v}=\left\{
\begin{array}{ll}
\binom{n-2k+i}{l-2k+i}& \mbox{if $|u\cap v|=i$ with  $\max\{2k-l,0\}\leq i\leq k$},\\
0& \mbox{otherwise,}\\
\end{array}
\right.
$$
or equivalently,
$$
BB^T=\sum_{i=\max\{2k-l,0\}}^{k}\binom{n-2k+i}{l-2k+i}\cdot A_i,
$$
 where $A_i$ is the adjacency matrix of  $J(n,k,i)$. Therefore, by Lemma \ref{Johnson-spec}, each eigenvalue of  $BB^T$ is of the form
$$\beta_s=\sum_{i=\max\{2k-l,0\}}^{k}\binom{n-2k+i}{l-2k+i}\cdot\alpha_{i,s}$$
with  multiplicity  $\binom{n}{s}-\binom{n}{s-1}$,  where  $\alpha_{i,s}$ is given in (\ref{equ-3}) and  $0\leq s\leq k$. Thus we obtain the spectrum of $G(n,k,l)$ immediately. 
\qed

\begin{remark}\label{rem-1}
\emph{
Since $G(n,k,l)$ is a  semi-regular bipartite graph with parameters $\left(\binom{n}{k},\binom{n}{l},\binom{n-k}{l-k},\binom{l}{k}\right)$, its  largest eigenvalue must be $\sqrt{\binom{n-k}{l-k}\binom{l}{k}}$.  We see that  this eigenvalue is exactly $\sqrt{\beta_0}$  in  Theorem \ref{set-spec-1}. In fact,  from  (\ref{equ-2}) we obtain 
$$
\begin{aligned}
\beta_0&=\sum_{i=\max\{2k-l,0\}}^{k}\binom{n-2k+i}{l-2k+i}\binom{k}{i}\binom{n-k}{k-i}\\
&=\sum_{i=\max\{2k-l,0\}}^{k}\frac{(n-k)!}{(l-2k+i)!(n-l)!(k-i)!}\binom{k}{i}\\
&=\sum_{i=\max\{2k-l,0\}}^{k}\frac{(n-k)!}{(n-l)!(l-k)!}\cdot \frac{(l-k)!}{(l-2k+i)!(k-i)!}\binom{k}{i}\\
&=\binom{n-k}{l-k}\sum_{i=\max\{2k-l,0\}}^{k}\binom{l-k}{k-i}\binom{k}{i}\\
&=\binom{n-k}{l-k}\binom{l}{k},
\end{aligned}
$$
as required.
}
\end{remark}

The following lemma gives the characteristic polynomial of the line graph of a semi-regular bipartite graph.
\begin{lem}\emph{(Cvetkovi\'{c}, Rowlinson and Simi\'{c} \cite[Corollary 2.4.3]{CRS})}\label{line-graph}
If $G$ is a semi-regular bipartite graph with parameters $(n_1,n_2,r_1,r_2)$ $(n_1\leq n_2)$, and if $\lambda_1\geq\lambda_2\geq\cdots\geq\lambda_{n_1}$ are the first $n_1$ largest eigenvalues, then
$$
\begin{aligned}
\phi(\mathcal{L}(G),x)&=(x-r_1-r_2+2)(x-r_2+2)^{n_2-n_1}(x+2)^{(n_2r_2-n_1-n_2+1)}\\
&~~\cdot\prod_{i=2}^{n_1}((x-r_1+2)(x-r_2+2)-\lambda_i^2).
\end{aligned}
$$
\end{lem}

\begin{cor}\label{set-spec-2}
Let $n$, $k$ and $l$ be integers with $1\leq k<l\leq n-1$ and $k+l\leq n$. Then the spectrum of the line graph $\mathcal{L}(G(n,k,l))$ is given by
$$
\begin{aligned}
&\left\{\left[\binom{n-k}{l-k}+\binom{l}{k}-2\right]^1,\left[\binom{l}{k}-2\right]^{\binom{n}{l}-\binom{n}{k}},[-2]^{\binom{n}{l}\binom{l}{k}-\binom{n}{k}-\binom{n}{l}+1}\right\}\\
&~~~~\cup \left\{[\gamma_{s,1}]^{\binom{n}{s}-\binom{n}{s-1}},[\gamma_{s,2}]^{\binom{n}{s}-\binom{n}{s-1}}\mid 1\leq s\leq k\right\},
\end{aligned}
$$
where $\gamma_{s,1}, \gamma_{s,2}$ are the two roots of  $\left(x-\binom{n-k}{l-k}+2\right)\cdot \left(x-\binom{l}{k}+2\right)-\beta_s=0$ with $\beta_s$  shown in (\ref{equ-2}).
\end{cor}

\begin{proof} 
Recall that $G(n,k,l)$ is a connected semi-regular bipartite graph with parameters $\left(\binom{n}{k},\binom{n}{l},\binom{n-k}{l-k},\binom{l}{k}\right)$, where $\binom{n}{k}\leq \binom{n}{l}$ due to $k<l$ and $k+l\leq n$. By Theorem \ref{set-spec-1}, the first $\binom{n}{k}$ largest eigenvalues of $G(n,k,l)$ are $\sqrt{\beta_s}$ (with multiplicity $\binom{n}{s}-\binom{n}{s-1}$), $s=0,1,\ldots,k$. In particular, the largest eigenvalue is $\sqrt{\beta_0}=\sqrt{\binom{n-k}{l-k}\binom{l}{k}}$ according to Remark \ref{rem-1}. By  Lemma \ref{line-graph},   $\mathcal{L}(G(n,k,l))$ has  characteristic polynomial 
\begin{equation*}
\begin{aligned}
&~~~~\phi(\mathcal{L}(G(n,k,l)),x)\\
&=\left(x-\binom{n-k}{l-k}-\binom{l}{k}+2\right)\cdot \left(x-\binom{l}{k}+2\right)^{\binom{n}{l}-\binom{n}{k}}\cdot (x+2)^{\binom{n}{l}\binom{l}{k}-\binom{n}{k}-\binom{n}{l}+1}\\
&~~~~\cdot\prod_{s=1}^{k}\left(\left(x-\binom{n-k}{l-k}+2\right)\cdot \left(x-\binom{l}{k}+2\right)-\beta_s\right)^{\binom{n}{s}-\binom{n}{s-1}},
\end{aligned}
\end{equation*}
and our result follows.
\end{proof}

By Theorem \ref{set-spec-1} and Corollary \ref{set-spec-2}, we can deduce the spectra of the  graph $B(n,k)$ and its line graph $\mathcal{L}(B(n,k))$, respectively.

\begin{cor}\label{layer-spec-1}
Let $n$ and $k$ be integers with  $1\leq k\leq \frac{n-1}{2}$. Then the spectrum of  $B(n,k)$ is
$$
\left\{\left[\pm\sqrt{(n-k-s)(k+1-s)}\right]^{\binom{n}{s}-\binom{n}{s-1}}\mid 0\leq s\leq k\right\}\cup \left\{\left[0\right]^{\binom{n}{k+1}-\binom{n}{k}}\right\}.
$$
\end{cor}

\begin{proof}
As $B(n,k)= G(n,k,k+1)$,  by setting $l=k+1$ in (\ref{equ-2}), we have
\begin{equation}\label{equ-5}
\begin{aligned}
\beta_s&=\sum_{i=k-1}^{k}\binom{n-2k+i}{i-k+1}\cdot\left(\sum_{r=0}^s(-1)^{s-r}\binom{s}{r}\binom{k-r}{i-r}\binom{n-k-s+r}{k-i-s+r}\right)\\
&=\sum_{r=0}^s(-1)^{s-r}\binom{s}{r}\binom{k-r}{k-r-1}\binom{n-k-s+r}{1-s+r}+\\
&~~~~(n-k)\cdot\sum_{r=0}^s(-1)^{s-r}\binom{s}{r}\binom{k-r}{k-r}\binom{n-k-s+r}{-s+r}\\
&=(n-k-s)(k+1-s),
\end{aligned}
\end{equation}
and the result follows from  Theorem \ref{set-spec-1} immediately.
\end{proof}

If $l=k+1$, then 
$\left(x-\binom{n-k}{l-k}+2\right)\cdot \left(x-\binom{l}{k}+2\right)-\beta_s=(x-s+2)(x-n+s+1)$ according to (\ref{equ-5}). By  Theorem \ref{set-spec-2}, we obtain the following result immediately.

\begin{cor}\label{layer-spec-2}
Let $n$ and $k$ be integers with  $1\leq k\leq \frac{n-1}{2}$. Then the line graph $\mathcal{L}(B(n,k))$ is an integral graph with  spectrum
$$
\begin{aligned}
&\left\{[n-1]^1,[k-1]^{\binom{n}{k+1}-\binom{n}{k}},[-2]^{k\binom{n}{k+1}-\binom{n}{k}+1}\right\}\\
&~~\cup \left\{[s-2]^{\binom{n}{s}-\binom{n}{s-1}},[n-s-1]^{\binom{n}{s}-\binom{n}{s-1}}\mid 1\leq s\leq k\right\}.
\end{aligned}
$$
\end{cor}

Putting $k=1$ in  Corollary \ref{layer-spec-2},  we obtain the main result of Mirafzal \cite{M1}.
\begin{cor}[Mirafzal \cite{M1}]\label{layer-spec-3}
Let $n\geq 4$ be an integer. Then the graph $\mathcal{L}(B(n,1))$ is an integral graph with distinct eigenvalues $-2$, $-1$, $0$, $n-2$, $n-1$.
\end{cor}

\section{The automorphism group  of $G(n,k,l)$}\label{s-3}

Let $S_n$ be the symmetric group on $[n]=\{1,2,\ldots,n\}$ with $n\geq 3$.  For each  $i$ with $1\leq i\leq n-1$, we denote by  $N_i$ the set of all $i$-subsets of $[n]$.  Suppose 
$$N=\cup_{i\in I}N_i,$$
where $I$ is some nonempty subset of $\{1,2,\ldots,n-1\}$. For any fixed $g\in S_n$, we define $\sigma_g$ as the mapping
\begin{equation}\label{equ-6}
\begin{aligned}
\sigma_g: N&\rightarrow N\\
v&\mapsto v^g=\{x^g\mid x\in v\}, \text{ for } v\in N.
\end{aligned}
\end{equation}
Clearly, $\sigma_g$ is one-to-one and onto. Denote by
\begin{equation}\label{equ-7}
\tilde{S}_n=\{\sigma_g\mid g\in S_n\}.
\end{equation}
It is not hard to verify that $\tilde{S}_n$ is  a permutation group on $N$ which is isomorphic to $S_n$. If the set $N=\cup_{i\in I}N_i$ satisfies $I=n-I=\{n-i\mid i\in I\}$, we  define
\begin{equation}\label{equ-8}
\begin{aligned}
\tau: N&\rightarrow N\\
v&\mapsto \bar{v}=[n]\setminus v, \text{ for } v\in N.
\end{aligned}
\end{equation}
It is clear that $\tau$ is a permutation  of order $2$ on $N$, and $\tau\not\in \tilde{S}_n$ due to $n\geq 3$.

By above arguments,  we see that $\tilde{S}_n$ is a permutation group on the vertex set $V=V_1\cup V_2=N_k\cup N_l$ of the set-inclusion graph $G(n,k,l)$, where $1\leq k<l\leq n-1$ and $k+l\leq n$. Let $\sigma_g\in \tilde{S}_n$.  For any two vertices $u,v\in V$, we have
$$
\begin{aligned}
\{u,v\}\in E&\Longleftrightarrow \mbox{$u\subset v$ or $v\subset u$}\\
&\Longleftrightarrow \mbox{$\sigma_g(u)=u^g\subset v^g=\sigma_g(v)$ or $\sigma_g(v)=v^g\subset u^g=\sigma_g(u)$}\\
&\Longleftrightarrow \mbox{$\{\sigma_g(u),\sigma_g(v)\}\in E$},
\end{aligned}
$$
which implies that $\sigma_g$ is an automorphism of $G(n,k,l)$. Thus we conclude that
$$\tilde{S}_n\leq \mathrm{Aut}(G(n,k,l)).$$
%In particular, it is easy to see that $G(n,k,l)$ is edge transitive.

Furthermore, if $l=n-k$, then $G(n,k,n-k)$ has vertex set $V=V_1\cup V_2=N_k\cup N_{n-k}$, and so the mapping $\tau$ defined in (\ref{equ-8}) is a permutation on $V$. For any two vertices $u,v\in V$, we obtain
$$
\begin{aligned}
\{u,v\}\in E&\Longleftrightarrow \mbox{$u\subset v$ or $v\subset u$}\\
&\Longleftrightarrow \mbox{$\tau(u)=\bar{u}\supset \bar{v}=\tau(v)$ or $\tau(v)=\bar{v}\supset \bar{u}=\tau(u)$}\\
&\Longleftrightarrow \mbox{$\{\tau(u),\tau(v)\}\in E$},
\end{aligned}
$$
implying that $\tau$ is also an automorphism of $G(n,k,n-k)$. As $1\leq k<l=n-k$, we have $n\geq 3$, and  $\tau\not\in \tilde{S}_n$. Additionally, one can  check that $\tau\cdot \sigma_g=\sigma_g\cdot \tau$ for all $\sigma_g\in \tilde{S}_n$. Therefore, we have $$\tilde{S}_n\times \langle\tau\rangle\leq \mathrm{Aut}(G(n,k,n-k)).$$

The result of Theorem \ref{set-aut-1} suggests that  the two  subgroups mentioned above are in fact the full automorphism group of $G(n,k,l)$  for $l<n-k$ and $l=n-k$, respectively. Before giving the proof, we need some powerful lemmas.

Recall that the Johnson graph $J(n,k)=J(n,k,k-1)$ ($1\leq k\leq \frac{n}{2}$) is the graph whose vertex set consists of all $k$-subsets of $[n]$, where two  vertices  are adjacent if their intersection has size $k-1$ (see Section \ref{s-2}).  As above, we see that $\mathrm{Aut}(J(n,k))$  contains $\tilde{S}_n$ as its subgroup, where  $\tilde{S}_n$ is given in (\ref{equ-7}). Indeed, we have

\begin{lem}\label{Johnson-aut}\emph{(Brouwer, Cohen and Neumaier \cite[Theorem 9.1.2]{BCN})}
Let $n$ and $k$ be  integers with  $1\leq k\leq \frac{n}{2}$. Then the automorphism group of the Johnson graph $J(n,k)$ is
$$
\mathrm{Aut}(J(n,k))=
\left\{
\begin{array}{ll}
\tilde{S}_n\times\langle\tau\rangle\cong S_n\times \mathbb{Z}_2 & \mbox{if $k=\frac{n}{2}\geq 2$},\\
\tilde{S}_n\cong S_n& \mbox{otherwise},
\end{array}
\right.
$$
where $\tilde{S}_n$ and $\tau$ are given in (\ref{equ-7}) and (\ref{equ-8}), respectively.
\end{lem}

The next lemma is straightforward.

\begin{lem}\label{bipartite}
Let $G$ be a connected bipartite graph with bipartition $V=V_1\cup V_2$. For each $\sigma\in \mathrm{Aut}(G)$,  we have  either $\sigma(V_1)=V_1$ and $\sigma(V_2)=V_2$, or $\sigma(V_1)=V_2$ and $\sigma(V_2)=V_1$.
\end{lem}

\begin{lem}\label{fix-aut}
Let $n$, $k$ and $l$ be integers with $1\leq k<l\leq n-1$ and $k+l\leq n$, and   let $V=V_1\cup V_2$ be the bipartition of $G(n,k,l)$ with $V_1$ and $V_2$ shown in (\ref{equ-1}). If $\sigma$ is an automorphism of $G(n,k,l)$ that fixes each vertex of $V_1$, then $\sigma$ is the identity automorphism.
\end{lem}

\begin{proof}
Let $v=\{x_1,x_2,\ldots,x_l\}$ be an arbitrary vertex of $V_2$. Set  $u_1=\{x_1,x_2,\ldots,x_k\},$
$u_2=\{x_2,\ldots,x_{k+1}\}, \ldots, u_{l-k+1}=\{x_{l-k+1},\ldots,x_{l}\}$. Clearly,  $u_1,u_2,\ldots, u_{l-k+1}\in V_1$. By the definition of $G(n,k,l)$, we see that $v$ is  the unique common neighbor of $u_1,u_2,\ldots,u_{l-k+1}$. Thus we must have $\sigma(v)=v$ because $u_1,u_2,\ldots,u_{l-k+1}$ have been fixed by $\sigma$ according to the assumption. By the arbitrariness of $v$, the result follows immediately.
\end{proof}

\begin{lem}\label{neighbor}
Let $n$, $k$ and $l$ be integers with $1\leq k<l\leq n-1$ and $k+l\leq n$, and   let $V=V_1\cup V_2$ be the bipartition of $G(n,k,l)$ with $V_1$ and $V_2$ shown in (\ref{equ-1}). Then for each $i$ with $\max\{2k-l,0\}\leq i\leq k-1$ and for any two  vertices $u,v\in V_1$, we have
$|u\cap v|=i$
if and only if
$|N(u)\cap N(v)|=\binom{n-2k+i}{l-2k+i}$.
\end{lem}
\begin{proof}
As $u,v\in V_1$, we have $|u|=|v|=k$. For $\max\{2k-l,0\}\leq i\leq k-1$,  if $|u\cap v|=i$ then we have $|u\cup v|=2k-i$, so  $u,v$ are simultaneously contained in $\binom{n-2k+i}{l-2k+i}$'s $l$-subsets of $[n]$, that is,  $|N(u)\cap N(v)|=\binom{n-2k+i}{l-2k+i}$.
This proves the necessity. Also, since $l\leq n-1$, for any $i,j$ with $\max\{2k-l,0\}\leq i\neq j\leq k-1$, we can verify that $\binom{n-2k+i}{l-2k+i}\neq \binom{n-2k+j}{l-2k+j}$. Thus the sufficiency follows.
\end{proof}

Now we give the proof of Theorem \ref{set-aut-1}.

{\flushleft \bf Proof of Theorem \ref{set-aut-1}.} 
At the beginning of this section, we have mentioned that $\tilde{S}_n$ and $\tilde{S}_n\times\langle\tau\rangle$ are the subgroups of $\mathrm{Aut}(G(n,k,l))$ for $k+l<n$ and $k+l=n$, respectively.

Let $V=V_1\cup V_2=N_k\cup N_l$  be the bipartition of $G(n,k,l)$. Let  $\sigma$ be an arbitrary automorphism of $G(n,k,l)$. By Lemma \ref{bipartite}, we only need to consider the following two cases.

{\flushleft\bf Case 1.} $\sigma(V_1)=V_1$ and $\sigma(V_2)=V_2$.

In this situation, let  $\sigma|_{V_1}$ be the restriction of $\sigma$ on $V_1$. Clearly, $\sigma|_{V_1}$ is a permutation on $V_1=N_k$. We have the following claim.
{\flushleft\bf Claim 1.} $\sigma|_{V_1}\in \mathrm{Aut}(J(n,k))=\tilde{S}_n$.

Recall that the Johnson graph $J(n,k)$ has $V_1=N_k=\{v\mid v\subseteq [n],~|v|=k\}$ as its vertex set, where  $u,v\in V_1$ are adjacent if and only if $|u\cap v|=k-1$. For any  $u,v\in V_1$, by Lemma \ref{neighbor}, we obtain
$$
\begin{aligned}
|u\cap v|=k-1&\Longleftrightarrow |N(u)\cap N(v)|=\binom{n-k-1}{l-k-1}\\
&\Longleftrightarrow |N(\sigma|_{V_1}(u))\cap N(\sigma|_{V_1}(v))|=\binom{n-k-1}{l-k-1}\\
&\Longleftrightarrow |\sigma|_{V_1}(u)\cap \sigma|_{V_1}(v)|=k-1,\\
\end{aligned}
$$
which  implies that $\sigma|_{V_1}\in \mathrm{Aut}(J(n,k))$. Also, by Lemma \ref{Johnson-aut}, we have $\mathrm{Aut}(J(n,k))=\tilde{S}_n$ because  $1\leq  k\leq \frac{n-1}{2}$. This proves Claim 1.

By Claim 1, there exists some $g\in S_n$ such that $\sigma|_{V_1}=\sigma_g$, where $\sigma_g$ is the permutation on $V_1=N_k$ defined in (\ref{equ-6}). Using the same $g$, we  can  extend $\sigma_g$ as a permutation on $V=V_1\cup V_2$ again by (\ref{equ-6}), and we still denote this permutation by $\sigma_g$. Clearly, $\sigma_g\in \tilde{S}_n\leq \mathrm{Aut}(G(n,k,l))$. Then we see that  $\sigma_g^{-1}\sigma$ fixes each vertex of $V_1$, and so must be the identity automorphism  by Lemma \ref{fix-aut}. Thus $\sigma=\sigma_g\in \tilde{S}_n$.

{\flushleft\bf Case 2.} $\sigma(V_1)=V_2$ and $\sigma(V_2)=V_1$.

In this situation, we must have $|V_1|=|V_2|$, and so $k+l=n$. Let $\tau$ be the automorphism of $G(n,k,n-k)$ defined in (\ref{equ-8}). We see that $\sigma\cdot \tau\in \mathrm{Aut}(G(n,k,n-k))$, and $\sigma\cdot \tau(V_i)=V_i$ for $i=1,2$.  According to what was proved in Case 1, there exists some $g\in S_n$ such that $\sigma\cdot \tau=\sigma_g$, i.e., $\sigma=\sigma_g\cdot \tau$. Therefore, we  conclude that $\sigma\in \tilde{S}_n\times \langle\tau\rangle$.

If $k+l<n$, only Case 1 can occur, thus we have $\mathrm{Aut}(G(n,k,l))\leq \tilde{S}_n\cong S_n$. If $k+l=n$, both Case 1 and Case 2 can occur, and so  $\mathrm{Aut}(G(n,k,n-k))\leq \tilde{S}_n \times \langle\tau\rangle\cong S_n\times \mathbb{Z}_2$, where $\mathbb{Z}_2$ is the cyclic group of order $2$. 

We complete the proof.
\qed

\begin{remark}
\emph{According to Theorem \ref{set-aut-1}, it is easy to see that the set-inclusion graph $G(n,k,l)$ is edge-transitive. Furthermore, for $l=n-k$ we claim that $G(n,k,n-k)$ is arc-transitive, since it is obviously vertex-transitive, and the stabilizer of any given vertex $v$ acts transitively on the set of neighbors. }
\end{remark}

Recall that $B(n,k)=G(n,k,k+1)$, and $H(n,k)=G(n,k,n-k)$. Putting $l=k+1$ and $n-k$  in Theorem \ref{set-aut-1}, we obtain the following two corollaries immediately.

\begin{cor}\label{layer-aut-1}\emph{(Mirafzal \cite{M2})}
Let $n$ and $k$ be integers with $1\leq k\leq \frac{n-1}{2}$. Then  the automorphism group of $B(n,k)$ is 
$$\mathrm{Aut}(B(n,k))\cong
\left\{
\begin{array}{ll}
S_n& \mbox{if $k<\frac{n-1}{2}$},\\
S_n\times \mathbb{Z}_2 & \mbox{if $k=\frac{n-1}{2}$},
\end{array}
\right.
$$
where $S_n$ is the symmetric group on $[n]$ and $\mathbb{Z}_2$ is the cycle group of order $2$.
\end{cor}

\begin{cor}\label{kneser-aut}\emph{(Mirafzal \cite{M1})}
Let $n$ and $k$ be integers with $1\leq k\leq \frac{n-1}{2}$. Then  the automorphism group of $H(n,k)$ is
$$\mathrm{Aut}(H(n,k))\cong S_n\times \mathbb{Z}_2,
$$
where $S_n$ is the symmetric group on $[n]$ and $\mathbb{Z}_2$ is the cycle group of order $2$.
\end{cor}

%The automorphism groups of the line graph and the original graph  are closely related. For any graph $G$, we define
%\begin{equation}\label{equ-5}
%\begin{aligned}
%h: \mathrm{Aut}(G)&\rightarrow \mathrm{Aut}(L(G))\\
%\sigma&\mapsto h(\sigma),
%\end{aligned}
%\end{equation}
%where $h(\sigma):V(L(G))\rightarrow V(L(G))$ is the mapping defined by
%$$
%h(\sigma)(\{u,v\})=\{\sigma(u),\sigma(v)\},~\{u,v\}\in E(G)=V(L(G)).
%$$
%It is well known that $h$ is indeed  an isomorphism if  $G$ is connected and not isomorphic to $K_2$, $K_4$,  $K_4$ with one edge deleted, or $K_4$ with two adjacent edges deleted (\cite[p. 120]{Biggs}).
%
%Let $L(n,k,l)$ be the line graph  of the set-inclusion graph $G(n,k,l)$. By above arguments, we have the follow corollary.
%\begin{cor}\label{cor-3}
%The automorphism group of $L(n,k,l)$ ($1\leq k<l\leq n$, $k+l\leq n$) is 
%$$\mathrm{Aut}(L(n,k,l))=
%\left\{
%\begin{array}{ll}
%h(\tilde{S}_n)\cong S_n& \mbox{if $k+l<n$},\\
%h(\tilde{S}_n\times\langle\tau\rangle)\cong S_n\times \mathbb{Z}_2 & \mbox{if $k+l=n$},
%\end{array}
%\right.
%$$
%where $\tilde{S}_n$ and $\tau$ are shown in (\ref{equ-3}) and (\ref{equ-4}), and $h$ is the mapping defined in \ref{equ-5}.
%\end{cor}

{\flushleft \bf Acknowledgements.} The first author is  supported  by the China Postdoctoral Science Foundation and the Postdoctoral Research Sponsorship in Henan Province (No. 1902011), the second author is  supported  by the National Natural Science Foundation of China (No. 11671344), and the third author is  supported by the National Natural Science Foundation of China (No. 11461054).

\end{document}